\documentclass[a4paper,10pt,reqno]{amsart}

\usepackage[english]{babel}
\usepackage[latin1]{inputenc}
\usepackage{amsmath,amssymb,amsfonts,amsthm,amscd}
\usepackage[mathscr]{eucal}
\usepackage{enumitem}
\usepackage{latexsym}
\usepackage{hyperref}

\newcommand{\df}{\mathrm{d}}

\newtheorem{theorem}{Theorem}[section]
\newtheorem{proposition}{Proposition}[section]

\newtheorem{lemma}{Lemma}[section]

\theoremstyle{definition}

\theoremstyle{remark}
  \newtheorem{remark}{Remark}[section]

\numberwithin{equation}{section}
\setlist[itemize]{parsep=0.4em}
\setlist[enumerate]{parsep=0.4em}
\title{Isoparametric surfaces in $\mathbb{E}(\kappa,\tau)$-spaces}
\date{}

\author[M.~Dom\'{\i}nguez-V\'{a}zquez]{Miguel Dom\'{\i}nguez-V\'{a}zquez}
\address{Instituto de Ciencias Matem\'aticas (CSIC-UAM-UC3M-UCM), Madrid, Spain.}
\email{miguel.dominguez@icmat.es}

\author{Jos\'{e} M. Manzano}
\address{Instituto de Ciencias Matem\'aticas (CSIC-UAM-UC3M-UCM), Madrid, Spain.}
\email{manzanoprego@gmail.com}

\thanks{The first author acknowledges support by the projects MTM2016-75897-P (AEI/FEDER, Spain), ED431F 2017/03 (Xunta de Galicia, Spain), as well as by the European Union's Horizon 2020 research and innovation programme under the Marie Sklodowska-Curie grant agreement No.~745722. 
The second author has been supported by the Spanish MCyT-Feder research project MTM2011-22547. 
Both authors have been supported by the project ICMAT Severo Ochoa project SEV-2015-0554 (MINECO, Spain).}



\begin{document}

\begin{abstract}
We provide an explicit classification of the following four families of surfaces in any homogeneous $3$-manifold with $4$-dimensional isometry group: isoparametric surfaces, surfaces with constant principal curvatures, homogeneous surfaces, and surfaces with constant mean curvature and vanishing Abresch-Rosenberg differential.
\end{abstract}

\maketitle

\section{Introduction}

A hypersurface of a Riemannian manifold is called isoparametric if it and its locally defined nearby equidistant hypersurfaces have constant mean curvature. In the 30s, Cartan characterized isoparametric hypersurfaces in space forms as those with constant principal curvatures, and achieved their classification in hyperbolic spaces $\mathbb{H}^n$. Segre obtained a similar result for Euclidean spaces $\mathbb{R}^n$. In both cases, isoparametric hypersurfaces are also open parts of extrinsically homogeneous hypersurfaces, that is, codimension one orbits of isometric actions on the ambient space. Surprisingly, in spheres $\mathbb{S}^n$ there are inhomogeneous examples and, indeed, the classification problem in spheres is much more involved and rich, giving rise to important recent contributions, such as~\cite{CCJ,Chi,Imm,Mi}.

In spaces of nonconstant curvature, very few classification results are known. Isoparametric hypersurfaces in complex hyperbolic spaces $\mathbb{C} \mathbb{H}^n$ and in $\mathbb{S}^2\times\mathbb{S}^2$ have been recently classified in~\cite{DDS} and~\cite{Ur}, respectively, and also in complex and quaternionic projective spaces, $\mathbb{C} \mathbb{P}^n$ and $\mathbb{H}  \mathbb{P}^n$, for almost all dimensions~\cite{Do,DG}. In most of these cases, and also in many Damek--Ricci harmonic spaces~\cite{DD}, there are isoparametric hypersurfaces with nonconstant principal curvatures which are, therefore, inhomogeneous, which contrasts with the situation in space forms. In general Riemannian manifolds, homogeneous hypersurfaces are always isoparametric and have constant principal curvatures, but none of the latter conditions necessarily implies homogeneity. Moreover, in spaces of nonconstant curvature, the isoparametric property and the constancy of the principal curvatures are, a priori, unrelated conditions. Hence, it is also interesting to address the classification problem of hypersurfaces with constant principal curvatures in specific Riemannian manifolds. This is a very complicated problem, even in ambient manifolds such as nonflat complex space forms $\mathbb{C} \mathbb{P}^n$ and $\mathbb{C} \mathbb{H}^n$~\cite{DD:indiana} or product spaces $\mathbb{S}^n\times\mathbb{R}$~or~$\mathbb{H}^n\times\mathbb{R}$~\cite{CS}.

The main purpose of this paper is to initiate the investigation of isoparametric surfaces and surfaces with constant principal curvatures in homogeneous $3$-manifolds. Simply connected homogeneous $3$-manifolds are classified attending to the dimension of their isometry group, which is equal to $3$, $4$ or $6$. If it is equal to $6$, then one obtains the space forms, but little is known about isoparametric surfaces if the dimension of the isometry group is $3$ or $4$. We will tackle the problem when the dimension is $4$, where such homogeneous $3$-manifolds are classified in the 2-parameter family of the so-called $\mathbb{E}(\kappa,\tau)$-spaces, with $\kappa-4\tau^2\neq 0$. The space $\mathbb{E}(\kappa,\tau)$ is the total space of a Riemannian submersion with bundle curvature $\tau$ over a simply connected complete surface $\mathbb{M}^2(\kappa)$ of constant curvature $\kappa$, see~\cite{Dan}. If $\tau=0$, one obtains the product spaces $\mathbb{H}^2(\kappa)\times\mathbb{R}$ and $\mathbb{S}^2(\kappa)\times\mathbb{R}$; if $\tau\neq 0$, then one has the Heisenberg space $\mathrm{Nil}_3$ ($\kappa=0$), the Berger spheres ($\kappa>0$), and the universal cover of the special linear group $\widetilde{\mathrm{SL}}_2(\mathbb{R})$ with some special left-invariant metrics~($\kappa<0$).   

Our main theorem provides the complete classification of isoparametric surfaces and surfaces with constant principal curvatures in $\mathbb{E}(\kappa,\tau)$-spaces of nonconstant curvature, and guarantees the homogeneity of such surfaces.

\begin{theorem}\label{thm:main}
Let $\Sigma$ be an immersed surface in $\mathbb{E}(\kappa,\tau)$, $\kappa-4\tau^2\neq 0$. The following assertions are equivalent:
\begin{enumerate}[label=(\roman*)]
	\item $\Sigma$ is an open subset of a homogeneous surface;
	\item $\Sigma$ is isoparametric;
	\item $\Sigma$ has constant principal curvatures;
	\item $\Sigma$ is an open subset of one of the following complete surfaces:
	\begin{enumerate}[label=(\alph*)]
	 \item a vertical cylinder over a complete curve of constant curvature in~$\mathbb{M}^2(\kappa)$, 
	 \item a horizontal slice $\mathbb{M}^2(\kappa)\times\{t_0\}$ with $\tau=0$,
	 \item a parabolic helicoid $P_{H,\kappa,\tau}$ with $4H^2+\kappa<0$.
	\end{enumerate}
\end{enumerate}
\end{theorem}

The surfaces $P_{H,\kappa,\tau}$ in item (iv)-(c) will be described in Section~\ref{sec:preliminaries}. According to Theorem~\ref{thm:main}, each one of the three families of examples in item (iv) is, therefore, an orbit of a cohomogeneity one isometric action on $\mathbb{E}(\kappa,\tau)$. Hence, it follows that every isoparametric surface (or every surface with constant principal curvatures) in $\mathbb{E}(\kappa,\tau)$ is an open part of an embedded, complete, isoparametric surface which, together with its equidistant surfaces, determines a codimension one singular Riemannian foliation on $\mathbb{E}(\kappa,\tau)$, see~\cite{AB}. This foliation is nothing but the orbit foliation induced by the cohomogeneity one isometric action mentioned above and, hence, its regular leaves are all isoparametric. Moreover, the foliations associated with examples (iv)-(b-c) do not have singular leaves (and, indeed, are invariant under vertical translations), whereas the foliations associated with the examples in (iv)-(a) have exactly one singular leaf if $4H^2+\kappa>0$ and no singular~leaves~otherwise.

The proof of Theorem~\ref{thm:main} will be scattered across the next sections in the following way. Firstly, the implication (i)$\Rightarrow$(ii) holds true for any hypersurface of any Riemannian manifold, since the existence of a group of ambient isometries acting with cohomogeneity one determines a family of equidistant hypersurfaces, all of them with constant principal curvatures and, hence, constant mean curvature. The implication (ii)$\Rightarrow$(iii) will be discussed in Section~\ref{sec:isoparametricas}, and follows from applying Jacobi field theory to work out the mean curvature of parallel surfaces (see also~\cite[\S10.2.1-2]{BCO}). The implication (iii)$\Rightarrow$(iv) is the content of Proposition~\ref{prop:cpc} in Section~\ref{sec:cpc}. Finally, (iv)$\Rightarrow$(i) follows from the discussion of examples in Section~\ref{sec:preliminaries}. 

We would like to mention that another equivalent property to those in the statement of Theorem~\ref{thm:main} is the condition that $\Sigma$ has constant mean curvature and constant angle function (see~\cite[Theorem~2.2]{ER}, where completeness is assumed but not used in the proof). It is also interesting to point out that Theorem~\ref{thm:main} implies that the Daniel sister correspondence~\cite{Dan} preserves isoparametric surfaces and surfaces with constant principal curvatures in $\mathbb{E}(\kappa,\tau)$.

In the classification of surfaces with constant principal curvatures, we will encounter the condition that the Abresch--Rosenberg differential $Q$ of the surface vanishes (see the definition in Section~\ref{sec:preliminaries}). On the one hand, Abresch and Rosenberg~\cite{AR} classified those surfaces with $Q=0$ when $\tau=0$ (an independent more geometric proof was also given by Leite~\cite{Leite}). On the other hand, Espinar and Rosenberg~\cite{ER} proved that, for any value of $\tau$, complete surfaces with $Q=0$ must be invariant under a 1-parameter group of ambient isometries. However a general explicit classification has not been established yet. Here we will provide such a classification in Proposition~\ref{prop:q=0}, whose proof will rely on~\cite{AR} along with the Daniel correspondence and a combinatorial argument.  

The proofs of Theorem~\ref{thm:main} and Proposition~\ref{prop:q=0} are purely local, not assuming completeness of the surfaces, so they also apply to surfaces in quotients of $\mathbb{E}(\kappa,\tau)$. Moreover, Jacobi fields in $\mathbb{E}(\kappa,\tau)$-spaces will be explicitly computed in Section~\ref{sec:isoparametricas}, though the particular case of the Heisenberg space can also be deduced from~\cite{BTV}.

It would be very interesting to extend Theorem~\ref{thm:main} to the case of homogeneous 3-manifolds with isometry group of dimension $3$, which are isometric to Lie groups with left-invariant metrics (a classification of such spaces can be found in~\cite{MP}). We expect that the families of isoparametric surfaces and surfaces with constant principal curvatures in these less symmetric spaces  will contain very few elements. Some homogeneous examples are given in~\cite{MS} as totally geodesic integral surfaces of certain distributions in very special cases, but a general classification of isoparametric surfaces or surfaces with constant principal curvatures is still an open problem.

\section{Preliminaries}\label{sec:preliminaries}

Given $\kappa,\tau\in\mathbb{R}$, let $\mathbb{M}^2(\kappa)$ be the simply connected complete surface of constant curvature $\kappa$. The oriented simply connected $3$-manifold $\mathbb{E}(\kappa,\tau)$ is characterized by admitting a Riemannian submersion $\pi\colon\mathbb{E}(\kappa,\tau)\to\mathbb{M}^2(\kappa)$ whose fibers are the integral curves of a unitary Killing vector field $\xi$, with bundle curvature $\tau$, i.e., for any tangent vector $v$, the following equation holds true
\begin{equation}\label{eqn:tau}
\overline\nabla_v\xi=\tau v\wedge \xi,
\end{equation}
where $\wedge$ stands for the cross-product in $\mathbb{E}(\kappa,\tau)$, whose sign depends on the chosen orientation. Then, the $2$-parameter family $\mathbb{E}(\kappa,\tau)$ with $\kappa-4\tau^2\neq 0$, parameterizes all simply connected homogeneous $3$-manifolds with $4$-dimensional isometry group. We refer to~\cite{Dan,Man} for details. A standard model~\cite{Dan} for $\mathbb{E}(\kappa,\tau)$ is given by
\[\mathbb{E}(\kappa,\tau)=\left\{(x,y,z)\in\mathbb{R}^3:1+\tfrac{\kappa}{4}(x^2+y^2)>0\right\}\]
endowed with the Riemannian metric
\[\frac{\df x^2+\df y^2}{(1+\frac{\kappa}{4}(x^2+y^2))^2}+\left(\df z^2+\frac{\tau\, (x\,\df y-y\,\df x)}{1+\frac{\kappa}{4}(x^2+y^2)}\right)^2.\]
This model omits a whole vertical fiber whenever $\kappa>0$, but this does not concern our arguments, which are purely local. In this model, the projection $\pi$ is nothing but $(x,y,z)\mapsto(x,y)$, whereas the unit Killing vector field is given by $\xi=\partial_z$.

The curvature tensor of $\mathbb{E}(\kappa,\tau)$ will be used below and is given by the following formula (see~\cite[Proposition~3.1]{Dan}, where the sign convention is the opposite of ours):
\begin{equation}\label{eq:curvature}
\begin{aligned}
\overline{R}(X,Y)Z={}&(\kappa-3\tau^2)(\langle Y,Z\rangle X-\langle X,Z\rangle Y) 
\\
&-(\kappa-4\tau^2)\left(\langle Y,\xi\rangle\langle Z,\xi\rangle X+\langle Y,Z\rangle\langle X,\xi\rangle \xi\right.
\\ 
& \left.\phantom{-(\kappa-4\tau^2)(}-\langle X,Z\rangle\langle Y,\xi\rangle \xi -\langle X,\xi\rangle\langle Z,\xi\rangle Y\right).
\end{aligned}
\end{equation}

Let now $\Sigma$ be an immersed surface in $\mathbb{E}(\kappa,\tau)$ with constant mean curvature $H$ with respect to a (local) unit normal $N$ (an $H$-surface in the sequel). Throughout the text, $H$ will be the average of the principal curvatures of $\Sigma$. We will denote by $\overline\nabla$ and $\nabla$ the Levi-Civita connections of $\mathbb{E}(\kappa,\tau)$ and $\Sigma$, respectively, and let $Au=-\overline\nabla_uN$ define the shape operator of $\Sigma$ for any vector $u$ tangent to $\Sigma$. In this setting, the intrinsic curvature $K$ of $\Sigma$ is given by the Gauss equation
\begin{equation}\label{eqn:gauss}
K=\det(A)+\tau^2+(\kappa-4\tau^2)\nu^2.
\end{equation}

Abresch and Rosenberg~\cite{AR} (cf.~\cite{ER}) considered the holomorphic quadratic differential $Q$ defined as the $(2,0)$-part of the $2$-form
\[\omega(u,v)=2(H+i\tau)\langle Au,v\rangle-(\kappa-4\tau^2)\langle u,\xi\rangle\langle v,\xi\rangle.\]
The modulus of $Q$ yields a geometric function given by $q(p)=\frac{1}{4}|Q_p(u,u)|^2$, see~\cite{Man2}. The function $q\in C^\infty(\Sigma)$ does not depend upon the choice of the unitary vector $u\in T_pM$, $p\in M$, and can be expressed in terms of the angle function $\nu=\langle N,\xi\rangle$, its gradient, and $\det(A)$ as follows (see~\cite[Lemma~2.2]{ER}):
\begin{equation}\label{eqn:q}
\frac{q}{\kappa-4\tau^2}=\left(\frac{4H^2+\kappa}{\kappa-4\tau^2}-\nu^2\right)\left(H^2-\det(A)+\tfrac{1}{4}(\kappa-4\tau^2)(1-\nu^2)\right)-\|\nabla\nu\|^2.
\end{equation}

Let us now discuss some distinguished examples that will show up in our classification results. On the one hand we have the vertical cylinders, which are defined as preimages $\Sigma=\pi^{-1}(\Gamma)$ of curves $\Gamma\subset\mathbb{M}^2(\kappa)$. It is easy to check that $\pi^{-1}(\Gamma)$ is an $H$-surface if and only if $\Gamma$ has constant curvature $2H$, which corresponds to examples in item (iv)-(a) of Theorem~\ref{thm:main}. These surfaces are invariant under vertical translations, as well as under a 1-parameter group of isometries of $\mathbb{E}(\kappa,\tau)$ leaving a horizontal lift of $\Gamma$ invariant, so they are homogeneous surfaces. From~\eqref{eqn:q} and the fact that $\det(A)=-\tau^2$, one easily gets that $q=(4H^2+\kappa)^2$ for vertical cylinders.

On the other hand, we have three families $S_{H,\kappa,\tau}$, $C_{H,\kappa,\tau}$, and $P_{H,\kappa,\tau}$ of surfaces whose Abresch--Rosenberg differential vanishes identically (this is a straightforward computation using Equation~\eqref{eqn:q} along with their explicit parameterizations):

\begin{itemize}
	 \item \textbf{Rotationally invariant surfaces $S_{H,\kappa,\tau}$.} In the above model for $\mathbb{E}(\kappa,\tau)$, consider the rotationally invariant surface parameterized by 
	 \begin{equation}\label{eqn:param-S}
	 X(u,v)=\left(v\cos(u),v\sin(u),\int_0^v\frac{-4Hs\sqrt{1+\tau^2s^2}\,\df s}{(4+\kappa s^2)\sqrt{1-H^2s^2}}\right),
	 \end{equation}
	 where $u\in\mathbb{R}$, and $v\in\bigl[0,\min\bigl\{\frac{1}{H},\frac{2}{\sqrt{-\kappa}}\bigr\}\bigr)$ if $\kappa<0$ or $v\in\bigl[0,\frac{1}{H}\bigr)$ if $\kappa\geq 0$. If $4H^2+\kappa>0$, this surface is the upper half of an $H$-sphere; otherwise, it is a complete $H$-surface everywhere transversal to $\xi$, i.e., an entire $H$-graph. The complete extension of~\eqref{eqn:param-S} will be denoted by $S_{H,\kappa,\tau}$. 

	 \item \textbf{Screw-motion invariant surfaces $C_{H,\kappa,\tau}$.} Assume that $4H^2+\kappa<0$, and consider the surface in $\mathbb{E}(\kappa,\tau)$ parameterized by
	 \begin{equation}\label{eqn:param-C}
	 \qquad X(u,v)\!=\!\left(v\cos(u),v\sin(u),\frac{4\tau}{\kappa}u\!\pm\!\!\int_{\frac{4H}{|\kappa|}}^v\!\frac{16H\sqrt{16\tau^2+\kappa^2s^2}\,\df s}{\kappa s(4+\kappa s^2)\sqrt{\kappa^2s^2-16H^2}}\right)\!,
	 \end{equation}
	 where $u\in\mathbb{R}$, and $v\in\bigl[\frac{4H}{-\kappa},\frac{2}{\sqrt{-\kappa}}\bigr)$. The condition $4H^2+\kappa<0$ ensures that~\eqref{eqn:param-C} defines a complete $H$-surface $C_{H,\kappa,\tau}$. If $\tau=0$, then $C_{H,\kappa,\tau}$ becomes a rotationally invariant surface that resembles a catenoid; if $H=0$, then $C_{H,\kappa,\tau}$ becomes a minimal helicoid of pitch $\frac{4\tau}{\kappa}$. The intermediate surfaces look like the deformation of the catenoid into the helicoid in $\mathbb{R}^3$, and they are not embedded.

	 \item \textbf{Parabolic helicoids $P_{H,\kappa,\tau}$.} Assume that $4H^2+\kappa< 0$, and consider the halfspace model $\mathbb{E}(\kappa,\tau)=\{(x,y,z)\in\mathbb{R}^3:y>0\}$ endowed with the Riemannian metric
	 \[\frac{\df x^2+\df y^2}{-\kappa y^2}+\left(\df z-\frac{2\tau}{\kappa y}\df x\right)^2.\]
	In this model, $P_{H,\kappa,\tau}$ is the entire $H$-graph parameterized by
	 \begin{equation}\label{eqn:param-P}
	 X(u,v)=\left(u,v,a\log(v)\right),\qquad\text{with } a=\frac{2H\sqrt{-\kappa+4\tau^2}}{-\kappa\sqrt{-4H^2-\kappa}}.
	 \end{equation}
	Observe that $P_{H,\kappa,\tau}$ is invariant under the 1-parameter groups of isometries $(x,y,z)\mapsto (x+t,y,z)$ and $(x,y,z)\mapsto(e^tx,e^ty,z-at)$, so it is homogeneous. Its angle function is constant and satisfies $\nu^2=\frac{4H^2+\kappa}{\kappa-4\tau^2}$. The homogeneous surfaces $P_{H,\kappa,\tau}$ are precisely those in item (iv)-(c) of Theorem~\ref{thm:main}.

	Parabolic helicoids are $2$-dimensional subgroups for a certain Lie group structure of $\mathbb{E}(\kappa,\tau)$, see~\cite{MP}. Moreover, the limit of $P_{H,\kappa,\tau}$ as $4H^2+\kappa\to 0$ is a vertical $H$-cylinder with $4H^2+\kappa=0$, and hence with $q=0$. These surfaces have already been studied by Leite~\cite{Leite} and Verpoort~\cite{V}.

\end{itemize}

\begin{remark}\label{rmk:q=0}
The examples arising in item (iv)-(b) of Theorem~\ref{thm:main} are horizontal slices $\mathbb{M}^2(\kappa)\times\mathbb\{t_0\}\subset\mathbb{M}^2(\kappa)\times\mathbb{R}$, which are clearly homogeneous. These can be recovered as $S_{0,\kappa,0}$ if $\kappa>0$, or $S_{0,\kappa,0}=C_{0,\kappa,0}=P_{0,\kappa,0}$ if $\kappa<0$ (note that $S_{0,\kappa,\tau}$ is the so-called horizontal umbrella centered at the origin). Otherwise, no two of the surfaces $S_{H,\kappa,\tau}$, $C_{H,\kappa,\tau}$, and $P_{H,\kappa,\tau}$ are congruent. This follows from the fact that isometries of $\mathbb{E}(\kappa,\tau)$ preserve the square of the angle function if $\kappa-4\tau^2\neq 0$: the surface $S_{H,\kappa,\tau}$ has points with $\nu^2=1$ and $\nu=0$, the surface $C_{H,\kappa,\tau}$ has points with $\nu=0$ but no points with $\nu^2=1$, and the surface $P_{H,\kappa,\tau}$ does not have any points with $\nu=0$ or $\nu^2=1$. Also, $P_{H,\kappa,\tau}$ has constant angle function, whereas $S_{H,\kappa,\tau}$ and $C_{H,\kappa,\tau}$ do not.
\end{remark}

\begin{remark}
Equation~\eqref{eqn:param-C} gives rise to a complete $H$-surface also for $4H^2+\kappa>0$, but it turns out to be congruent to $S_{H,\kappa,\tau}$. This is because rotationally invariant surfaces in $\mathbb{S}^2\times\mathbb{R}$ or in the Berger spheres automatically become invariant by screw motions with respect to the antipodal axis. This also yields a geometric reason why the coefficient of $u$ in the third component of~\eqref{eqn:param-C} is precisely $\frac{4\tau}{\kappa}$. 
\end{remark}

Next we will classify explicitly those $H$-surfaces with $q=0$. Espinar and Rosenberg~\cite{ER} showed that they must be invariant under a $1$-parameter group of ambient isometries. Our argument is independent of~\cite{ER}, reasoning directly from Abresch and Rosenberg's classification in the case $\tau=0$ by means of the Daniel sister correspondence~\cite{Dan}. Observe that Abresch and Rosenberg's surfaces $D_H$ and $S_H$ are nothing but our $S_{H,\kappa,0}$ ($D_H$ corresponds to $4H^2+\kappa\leq 0$ whereas $S_H$ corresponds to $4H^2+\kappa >0$), their $C_H$ are our $C_{H,\kappa,0}$, and their $P_H$ are our $P_{H,\kappa,0}$ plus the vertical $H$-cylinder in $\mathbb{H}^2(\kappa)\times\mathbb{R}$ if $4H^2+\kappa=0$.

\begin{proposition}\label{prop:q=0}
Let $\Sigma$ be an immersed $H$-surface in $\mathbb{E}(\kappa,\tau)$, $\kappa-4\tau^2\neq 0$, with vanishing Abresch--Rosenberg differential. Then $\Sigma$ is congruent to an open subset of one of the following complete surfaces:
\begin{enumerate}[label=(\alph*)]
 \item the rotationally invariant surface $S_{H,\kappa,\tau}$,
 \item the screw-motion invariant surface $C_{H,\kappa,\tau}$ with $\kappa-4\tau^2<0$,
 \item the parabolic helicoid $P_{H,\kappa,\tau}$ with $4H^2+\kappa<0$,
 \item the vertical $H$-cylinder with $4H^2+\kappa=0$.
\end{enumerate}
\end{proposition}

\begin{proof}
Let $\Sigma^*$ be the sister immersed $H^*$-surface in $\mathbb{E}(\kappa-4\tau^2,0)$, where $H^*=\sqrt{H^2+\tau^2}$. Since the quantities $4H^2+\kappa$ and $\kappa-4\tau^2$, the angle function, and the intrinsic geometry are preserved by the correspondence, it follows from~\eqref{eqn:gauss} and~\eqref{eqn:q} that the surface $\Sigma^*$ also satisfies $q=0$ in $\mathbb{E}(\kappa-4\tau^2,0)$. We can assume that $\Sigma^*$ is complete since Abresch and Rosenberg do not use completeness in their arguments (see also~\cite{Leite}) showing that the surface is an open subset of a complete surface. Their classification~\cite[Theorem~3]{AR} produces (up to ambient isometries) one $H^*$-surface in $\mathbb{E}(\kappa-4\tau^2,0)$ if $4H^2+\kappa>0$, two of them if $4H^2+\kappa=0$, and three of them if $4H^2+\kappa<0$.

Via the correspondence, there must be the same number of $H$-surfaces in $\mathbb{E}(\kappa,\tau)$ with $q=0$ up to ambient isometries. It is straightforward to check that the $H$-surfaces in the above items (a)-(d) satisfy $q=0$, so they must be the only ones (see also Remark~\ref{rmk:q=0}).
\end{proof}

\begin{remark}
By the analysis of the angle function in Remark~\ref{rmk:q=0}, it is easy to show that the sister correspondence leaves each of the families $S_{H,\kappa,\tau}$, $C_{H,\kappa,\tau}$, and $P_{H,\kappa,\tau}$ invariant. The parameterizations given by~\eqref{eqn:param-S},~\eqref{eqn:param-C} and~\eqref{eqn:param-P} have been found among the invariant surfaces by imposing the condition $q=0$, inspired by~\cite{ER}.
\end{remark}

\section{Surfaces with constant principal curvatures}\label{sec:cpc}

Let $\Sigma$ be a surface with constant principal curvatures in $\mathbb{E}(\kappa,\tau)$, $\kappa-4\tau^2\neq 0$, and let $\{e_1,e_2\}$ be a local orthonormal frame diagonalizing its shape operator $A$ with respect to a (local) unit normal $N$, i.e., $Ae_i=\kappa_ie_i$ for some real constants $\kappa_1,\kappa_2\in\mathbb{R}$. Since $\{e_1,e_2\}$ is orthonormal, there exist smooth functions $p_1$, $p_2$, called the Christoffel symbols associated with this frame, such that
\begin{equation}\label{eqn:christoffel1}
\begin{aligned}
\nabla_{e_1}e_1&=p_1e_2,&\nabla_{e_1}e_2&=-p_1e_1,\\
\nabla_{e_2}e_1&=p_2e_2,&\nabla_{e_2}e_2&=-p_2e_1.
\end{aligned}\end{equation}
The curvature tensor of $\mathbb{E}(\kappa,\tau)$ applied to the triple $\{e_1,e_2,N\}$ can be computed easily by means of~\eqref{eqn:christoffel1} and the fact that $[e_1,e_2]=\nabla_{e_1}e_2-\nabla_{e_2}e_1$, as
\begin{equation}\label{eqn:curvature1}
 \overline{R}(e_1,e_2)N=\overline\nabla_{e_1}\overline\nabla_{e_2}N-\overline\nabla_{e_2}\overline\nabla_{e_1}N-\overline\nabla_{[e_1,e_2]}N=(\kappa_2-\kappa_1)(p_1e_1-p_2e_2).
\end{equation}
It can be also computed by means of~\eqref{eq:curvature}, so Equation~\eqref{eqn:curvature1} can be rewritten as
\begin{equation}\label{eqn:christoffel2}
\begin{aligned}
 (\kappa_1-\kappa_2)p_1=(\kappa-4\tau^2)\nu\langle T,e_2\rangle,\\
 (\kappa_1-\kappa_2)p_2=(\kappa-4\tau^2)\nu\langle T,e_1\rangle,
\end{aligned}
\end{equation}
where $\nu=\langle N,\xi\rangle$ denotes the angle function, and $T=\xi-\nu N$ is the tangent part of the unit Killing vector field $\xi$. Two formulas that will come in handy later are 
\begin{align}
 \nabla_vT&=\tau\nu Jv+\nu Av,\label{eqn:nablaT}\\
 \nabla\nu&=-AT+\tau JT,\label{eqn:nablaNu}
\end{align}
where $v$ is any tangent vector. Here $J$ denotes a $\frac{\pi}{2}$-rotation in the tangent plane, which can be defined in terms of the cross-product as $Jv=v\wedge N$. Moreover, we will assume that $\{e_1,e_2\}$ is such that $Je_1=e_2$ and $Je_2=-e_1$.

\begin{proposition}\label{prop:cpc}
Let $\Sigma$ be an immersed surface with constant principal curvatures in $\mathbb{E}(\kappa,\tau)$, $\kappa-4\tau^2\neq 0$. Then $\Sigma$ is an open subset of one of the following surfaces:
	\begin{enumerate}[label=(\alph*)]
	 \item a vertical cylinder over a complete curve of constant curvature in $\mathbb{M}^2(\kappa)$, 
	 \item a horizontal slice $\mathbb{M}^2(\kappa)\times\{t_0\}$ with $\tau=0$,
	 \item the parabolic helicoid $P_{H,\kappa,\tau}$ with $4H^2+\kappa<0$.	 
	\end{enumerate}
\end{proposition}

\begin{proof}
In the totally umbilical case $\kappa_1=\kappa_2$, it follows from~\eqref{eqn:christoffel2} that either $\nu=0$ or $T=0$. Since the only totally umbilical surfaces with $\nu=0$ are the vertical cylinders over geodesics in product spaces, and the only surfaces with $T=0$ are horizontal slices, we get that $\tau=0$ and $\Sigma$ is totally geodesic either way. This is essentially the argument in the classification of totally umbilical surfaces of Souam and Toubiana~\cite{ST}. 

Therefore, we will assume in the sequel that $\kappa_1\neq\kappa_2$, and write~\eqref{eqn:christoffel2} as $p_1=\sigma\nu\langle T,e_2\rangle$ and $p_2=\sigma\nu\langle T,e_1\rangle$, where
\[\sigma=\frac{\kappa-4\tau^2}{\kappa_1-\kappa_2}.\]
Since the set of points where $\nu=0$ or $\nu^2=1$ has empty interior unless $\Sigma$ is a vertical cylinder or a totally geodesic slice with $\tau=0$, and our computation is local, we will also assume that both $\nu$ and $T$ do not vanish.

The Gaussian curvature of $\Sigma$ can be computed from~\eqref{eqn:christoffel1} as
\begin{equation}\label{lemma:cpc:eqn1}
K=\langle\nabla_{e_1}\nabla_{e_2}e_2-\nabla_{e_2}\nabla_{e_1}e_2-\nabla_{[e_1,e_2]}e_2,e_1\rangle=e_2(p_1)-e_1(p_2)-p_1^2-p_2^2.
\end{equation}

On the one hand, we can work out the derivatives of $\sigma^{-1}p_1=\nu\langle T,e_2\rangle$ and $\sigma^{-1}p_2=\nu\langle T,e_1\rangle$ by means of~\eqref{eqn:christoffel1},~\eqref{eqn:christoffel2},~\eqref{eqn:nablaT}, and~\eqref{eqn:nablaNu}, giving rise to
\begin{equation}\label{lemma:cpc:eqn3}
\begin{aligned}
 \sigma^{-1}e_2(p_1)&=\langle e_2,\nabla\nu\rangle\langle T,e_2\rangle+\nu\langle\nabla_{e_2}T,e_2\rangle+\nu\langle T,\nabla_{e_2}e_2\rangle\\
 &=\kappa_2\nu^2-\sigma\nu^2\langle T,e_1\rangle^2+\tau\langle T,e_1\rangle\langle T,e_2\rangle-\kappa_2\langle T,e_2\rangle^2,\\
 \sigma^{-1}e_1(p_2)&=\langle e_1,\nabla\nu\rangle\langle T,e_1\rangle+\nu\langle\nabla_{e_1}T,e_1\rangle+\nu\langle T,\nabla_{e_1}e_1\rangle\\
 &=\kappa_1\nu^2-\kappa_1\langle T,e_1\rangle^2-\tau\langle T,e_1\rangle\langle T,e_2\rangle+\sigma\nu^2\langle T,e_2\rangle^2.
\end{aligned}\end{equation}
On the other hand, using the identity $\|T\|^2=\|\xi-\nu N\|^2=1-\nu^2$, we reach
\begin{equation}\label{lemma:cpc:eqn4}
p_1^2+p_2^2=\sigma^2\nu^2(\langle T,e_1\rangle^2+\langle T,e_2\rangle^2)=\sigma^2\nu^2\|T\|^2=\sigma^2\nu^2(1-\nu^2).
\end{equation}
Therefore, plugging~\eqref{eqn:gauss},~\eqref{lemma:cpc:eqn3}, and~\eqref{lemma:cpc:eqn4} into~\eqref{lemma:cpc:eqn1}, we get the relation
\begin{equation}\label{lemma:cpc:eqn5}
\begin{aligned}
\kappa_1\langle T,e_1\rangle^2+2\tau\langle T,e_1\rangle\langle T,e_2\rangle-\kappa_2\langle T,e_2\rangle^2&=\sigma^{-1}(\kappa_1\kappa_2+\tau^2)\\
&\quad+2\nu^2(\sigma(1-\nu^2)+\kappa_1-\kappa_2).
\end{aligned}
\end{equation}
If $\kappa_1=-\kappa_2$ and $\tau=0$, the left-hand side of~\eqref{lemma:cpc:eqn5} is equal to $\kappa_1(\langle T,e_1\rangle^2+\langle T,e_2\rangle^2)=\kappa_1(1-\nu^2)$, and~\eqref{lemma:cpc:eqn5} becomes a polynomial of degree $4$ in $\nu$ with constant coefficients and leading coefficient $2\sigma\neq 0$. Thus, we deduce that $\nu$ must be constant, so~\eqref{eqn:christoffel2} and~\eqref{eqn:nablaNu} yield $0=-\sigma\nu\nabla\nu=p_2\kappa_1e_1+p_1\kappa_2e_2$. Hence, either $\kappa_1=-\kappa_2=0$ or $p_1=p_2=0$, which, by~\eqref{eqn:christoffel2}, leads to vertical cylinders or totally geodesic surfaces, which have been already ruled out.

Therefore, we will assume that $\kappa_1\neq-\kappa_2$ or $\tau\neq0$. Then~\eqref{lemma:cpc:eqn5} together with $\langle T,e_1\rangle^2+\langle T,e_2\rangle^2=1-\nu^2$ form a system of two quadratic equations with unknowns $\langle T,e_1\rangle$ and $\langle T,e_2\rangle$. Since $\kappa_1\neq-\kappa_2$ or $\tau\neq0$, this system has finitely many solutions, so both $\langle T,e_1\rangle$ and $\langle T,e_2\rangle$ must be functions of $\nu$. In particular, the identities $\langle\nabla\langle T,e_i\rangle,J\nabla\nu\rangle=0$, $i\in\{1,2\}$, hold. On the one hand, from~\eqref{eqn:nablaNu} we infer that
\begin{equation}\label{lemma:cpc:eqn6}
J\nabla\nu=-JAT-\tau T=(\kappa_2 \langle T,e_2\rangle-\tau \langle T,e_1\rangle)e_1-(\kappa_1\langle T,e_1\rangle+\tau \langle T,e_2\rangle)e_2.
\end{equation}
On the other hand, using~\eqref{eqn:christoffel1},~\eqref{eqn:christoffel2}, and~\eqref{eqn:nablaT}, it is easy to compute
\begin{equation}\label{lemma:cpc:eqn7}
\begin{aligned}
\nabla\langle T,e_1\rangle&=(\langle\nabla_{e_1}T,e_1\rangle+\langle T,\nabla_{e_1}e_1\rangle)e_1+(\langle\nabla_{e_2}T,e_1\rangle+\langle T,\nabla_{e_2}e_1\rangle)e_2\\
&=(\kappa_1+\sigma\langle T,e_2\rangle^2)\nu e_1+(-\tau+\sigma\langle T,e_1\rangle\langle T,e_2\rangle)\nu e_2,\\
\nabla\langle T,e_2\rangle&=(\langle\nabla_{e_1}T,e_2\rangle+\langle T,\nabla_{e_1}e_2\rangle)e_1+(\langle\nabla_{e_2}T,e_2\rangle+\langle T,\nabla_{e_2}e_2\rangle)e_2\\
&=(\tau-\sigma\langle T,e_1\rangle\langle T,e_2\rangle)\nu e_1+(\kappa_2-\sigma\langle T,e_1\rangle^2)\nu e_2.
\end{aligned}
\end{equation}
In view of~\eqref{lemma:cpc:eqn6} and~\eqref{lemma:cpc:eqn7}, the equations $\langle\nabla\langle T,e_i\rangle,J\nabla\nu\rangle=0$, $i\in\{1,2\}$, can be written, respectively, as
\begin{equation}\label{lemma:cpc:eqn8}
\begin{aligned}
0&=\langle T,e_2\rangle(\kappa_1\kappa_2+\tau^2-\sigma(\kappa_1\langle T,e_1\rangle^2+2\tau\langle T,e_1\rangle\langle T,e_2\rangle-\kappa_2\langle T,e_2\rangle^2)), \\
0&=\langle T,e_1\rangle(\kappa_1\kappa_2+\tau^2-\sigma(\kappa_1\langle T,e_1\rangle^2+2\tau\langle T,e_1\rangle\langle T,e_2\rangle-\kappa_2\langle T,e_2\rangle^2)).
\end{aligned}
\end{equation}
Since we are assuming that $T\neq 0$, we deduce that
\[\kappa_1\langle T,e_1\rangle^2+2\tau\langle T,e_1\rangle\langle T,e_2\rangle-\kappa_2\langle T,e_2\rangle^2=\sigma^{-1}(\kappa_1\kappa_2+\tau^2),\]
so that~\eqref{lemma:cpc:eqn5}, together with the fact that $\nu\neq 0$, leads to the simplified expression
\begin{equation}\label{lemma:cpc:eqn10}
\sigma(1-\nu^2)+\kappa_1-\kappa_2=0.
\end{equation}
In particular, it follows that $\nu$ is constant, so~\eqref{eqn:nablaNu} yields 
\begin{equation}\label{lemma:cpc:eqn9}
\begin{aligned}
0&=\langle\nabla\nu,e_1\rangle=-\kappa_1\langle T,e_1\rangle-\tau\langle T,e_2\rangle,\\
0&=\langle\nabla\nu,e_2\rangle=\tau\langle T,e_1\rangle-\kappa_2\langle T,e_2\rangle.
\end{aligned}
\end{equation}
Considering~\eqref{lemma:cpc:eqn9} as a linear system with unknowns $\langle T,e_1\rangle$ and $\langle T,e_2\rangle$, it follows from $T\neq 0$ that its determinant $\kappa_1\kappa_2+\tau^2$ must vanish. Hence~\eqref{lemma:cpc:eqn10} gives
\[\nu^2=1+\sigma^{-1}(\kappa_1-\kappa_2)=1+\frac{(\kappa_1-\kappa_2)^2}{\kappa-4\tau^2}=1+\frac{4H^2+4\tau^2}{\kappa-4\tau^2}=\frac{4H^2+\kappa}{\kappa-4\tau^2},\]
and~\eqref{eqn:q} implies that $q=0$. Since the surfaces $S_{H,\kappa,\tau}$ and $C_{H,\kappa,\tau}$ do not have constant angle function, and vertical $H$-cylinders have been discarded previously, we get from Proposition~\ref{prop:q=0} that $\Sigma$ must be an open subset of $P_{H,\kappa,\tau}$.
\end{proof}

\section{Isoparametric surfaces}\label{sec:isoparametricas}

In this section we prove that any isoparametric surface in $\mathbb{E}(\kappa,\tau)$ has constant principal curvatures and constant angle function. For this purpose we will use Jacobi field theory to analyze the extrinsic geometry of  equidistant surfaces to a given one. We refer to~\cite[\S10.2.1-2]{BCO} for more details on this method.

Let $\Sigma$ be an isoparametric surface in $\mathbb{E}(\kappa,\tau)$ with (locally defined) unit normal $N$. We can restrict to an open subset where the angle function satisfies $\nu^2\neq 1$, since otherwise $\Sigma$ would be an open part of a horizontal slice with $\tau=0$, which is a homogeneous surface and, hence, has constant principal curvatures.
	
For each $r\in \mathbb{R}$ we define the map $\Phi^r\colon \Sigma\to \mathbb{E}(\kappa,\tau)$, $p\mapsto\exp_p(rN_p)$, where $\exp_p$ is the Riemannian exponential map of $\mathbb{E}(\kappa,\tau)$ at $p$. Hence, $\Phi^r(\Sigma)$ is obtained by moving $\Sigma$ a distance $r$ in the normal direction. For each $p\in\Sigma$, let us denote by $\gamma_p$ the geodesic in $\mathbb{E}(\kappa,\tau)$ with initial conditions $\gamma_p(0)=p$ and $\gamma_p'(0)=N_p$. Thus, locally, there exists $\varepsilon>0$ such that, if $r\in(-\varepsilon,\varepsilon)$, then $\Sigma^r=\Phi^r(\Sigma)$ is an embedded equidistant surface to $\Sigma$, and $\gamma_p'(r)$ is a normal vector to $\Sigma^r$. Hence, we restrict $\Sigma$ if necessary so that  $\mathcal{U}=\cup_{r\in(-\varepsilon,\varepsilon)}\Sigma^r$ is an open subset of $\mathbb{E}(\kappa,\tau)$ foliated by the equidistant surfaces $\Sigma^r$. We extend the unit normal $N$ to $\mathcal{U}$ by $N_{\gamma_p(r)}=\gamma_p(r)$, for $p\in \Sigma$, $r\in(-\varepsilon,\varepsilon)$. We also extend the angle function $\nu=\langle N,\xi\rangle$ to $\mathcal{U}$, and the complex structure $J$ to each surface $\Sigma^r\subset\mathcal{U}$ as $Jv=v\wedge N$ (for any vector $v\in T\mathcal{U}$ tangent to any of the $\Sigma^r$). Since $N$ is a geodesic vector field and $\xi$ is Killing, it follows that $\nu$ is constant along $\gamma_p$, for all $p\in\Sigma$. 

Observe that $\nu^2\neq 1$ at each point of $\mathcal{U}$, so we can define 
\begin{equation}\label{eqn:U1U2}
U_1=\frac{\xi-\nu N}{\sqrt{1-\nu^2}} \qquad \text{and}\qquad U_2=JU_1,
\end{equation}
and $\{U_1,U_2,N\}$ is an orthonormal frame on $\mathcal{U}$, where $U_1$, $U_2$ are tangent to the surfaces $\Sigma^r$. The following derivatives will be useful in the computations below. 

\begin{lemma}
We have that
\begin{equation}\label{eq:nabla_u_i}
\overline{\nabla}_{N} U_1=\tau U_2\qquad \text{and}\qquad \overline{\nabla}_N U_2=-\tau U_1.
\end{equation}
\end{lemma}

\begin{proof}
From~\eqref{eqn:U1U2} and~\eqref{eqn:tau}, along with the fact that $\nu$ is constant in the direction of $N$, as well as the fact that $\overline\nabla_NN=0$ ($N$ is a geodesic vector field), we get that
\[\overline\nabla_NU_1=\overline\nabla_N\frac{\xi-\nu N}{\sqrt{1-\nu^2}}=\frac{\overline\nabla_N\xi}{\sqrt{1-\nu^2}}=\frac{\tau N\wedge\xi}{\sqrt{1-\nu^2}}=\tau N\wedge\frac{\xi-\nu N}{\sqrt{1-\nu^2}}=\tau U_2.\]
The second equation in~\eqref{eq:nabla_u_i} follows from the first one and the product rule 
\[\overline\nabla_N(U_1\wedge N)=\overline\nabla_NU_1\wedge N+U_1\wedge\overline\nabla_NN=\tau U_2\wedge N=-\tau U_1.\qedhere\]
\end{proof}

For each normal geodesic $\gamma_p$ in $\mathcal{U}$ we are interested in calculating certain Jacobi fields along $\gamma_p$. Recall that a Jacobi field along $\gamma_p$ is a vector field $\zeta$ along $\gamma_p$ satisfying the Jacobi equation $\zeta''+\overline{R}(\zeta,\gamma_p')\gamma_p'=0$, where $'$ denotes $\overline{\nabla}$-covariant derivative along $\gamma_p$. For each $j\in\{1,2\}$, let us consider the Jacobi field $\zeta_j$ along $\gamma_p$ with initial conditions 
\begin{equation}\label{eq:in_cond}
\zeta_j(0)=U_j(p) \qquad \text{and}\qquad \zeta_j'(0)=-AU_j(p),
\end{equation}
where $A$ is the shape operator of~$\Sigma$. Since these initial conditions are orthogonal to $\gamma_p'(0)$, the Jacobi field $\zeta_j$ is also orthogonal to $\gamma_p'$ and, hence, can be written as $\zeta_j=b_{1j}U_1+b_{2j}U_2$, where $b_{ij}$ are certain smooth functions on $(-\varepsilon,\varepsilon)$. Using the relations~\eqref{eq:nabla_u_i} and the fact that $\gamma_p'=N$ along $\gamma_p$, we have
\begin{align}\label{eq:zeta''}\nonumber
\zeta_j''&=\overline{\nabla}_{\gamma_p'}\overline{\nabla}_{\gamma_p'}\zeta_j=\overline{\nabla}_{\gamma_p'}\bigl((b_{1j}'-\tau b_{2j})U_1+(b_{2j}'+\tau b_{1j})U_2\bigr)
\\
&=(b_{1j}''-2\tau b_{2j}'-\tau^2 b_{1j})U_1+(b_{2j}''+2\tau b_{1j}'-\tau^2 b_{2j})U_2.
\end{align}
The expression of $\overline{R}$ given by~\eqref{eq:curvature} lets us calculate
\begin{align}\label{eq:R_jacobi}\nonumber
\overline{R}(\zeta_j,\gamma_p')\gamma_p'&=(\kappa-3\tau^2)\zeta_j-(\kappa-4\tau^2)\bigl(\nu^2\zeta_j+\langle \zeta_j,\xi\rangle\xi-\langle \zeta_j,\xi\rangle \nu N\bigr)
\\\nonumber
&=(\kappa-3\tau^2)\zeta_j-(\kappa-4\tau^2)(\nu^2\zeta_j+b_{1j}(1-\nu^2)U_1)
\\
&= \tau^2 b_{1j} U_1+\bigl((\kappa-4\tau^2)(1-\nu^2)+\tau^2\bigr)b_{2j} U_2.
\end{align}
Therefore, the $\zeta_j$ are Jacobi fields provided that the $b_{ij}$ satisfy the following homogeneous linear system of ordinary differential equations with constant coefficients (recall that $\nu$ is constant along $\gamma_p$)
\begin{equation}\label{eq:ode}
b_{1j}''-2\tau b_{2j}'=0, \qquad b_{2j}''+2\tau b_{1j}'+(\kappa-4\tau^2)(1-\nu^2)b_{2j}=0.
\end{equation}
Let the shape operator of $\Sigma$ be determined by the relations $AU_i=a_{i1}U_1+a_{i2}U_2$, $i\in\{1,2\}$, for certain smooth functions $a_{ij}$ with $a_{21}=a_{12}$. Then, taking into account that $\zeta_j'=\overline{\nabla}_{N}\zeta_j$ and \eqref{eq:nabla_u_i}, the initial conditions~\eqref{eq:in_cond} of $\zeta_j$ are equivalent to the relations
\begin{equation}\label{eq:initial_conditions}
\begin{gathered}
b_{11}(0)=b_{22}(0)=1, \quad b_{12}(0)=b_{21}(0)=0,
\\
\quad b_{11}'(0)=-a_{11}, \quad b_{22}'(0)=-a_{22},\quad b_{12}'(0)=\tau-a_{12}, \quad b_{21}'(0)=-\tau-a_{12}.
\end{gathered}
\end{equation}
For convenience, we define a function $\delta$ on $\Sigma$ by 
\begin{equation}\label{eq:delta}
\delta=(\kappa-4\tau^2)\nu^2-\kappa,
\end{equation}
In the sequel we will assume that $\delta\neq 0$ on $\Sigma$ by restricting to an open subset if necessary. Observe that, if $\delta=0$ on an open subset of $\Sigma$, then $\nu$ is constant on this subset, so it follows from the arguments below~\eqref{lemma:cpc:eqn9} (also from~\cite{ER}) that $\Sigma$ is one of the examples in item (iv) of Theorem~\ref{thm:main}, so it has constant principal curvatures. 

\begin{remark}
Since $\kappa-4\tau^2\neq 0$, we get that $\delta=-\kappa(1-\nu^2)-4\tau^2\nu^2\leq0$ provided that $\kappa\geq 0$. However, if $\kappa<0$, the sign of $\delta$ might vary from point to point, but this has a geometric interpretation: $\delta(p)>0$ when $\gamma_p$ projects to a curve of $\mathbb{H}^2(\kappa)$ which is equidistant to a geodesic, $\delta(p)=0$ when $\gamma_p$ projects onto a horocycle of $\mathbb{H}^2(\kappa)$, and $\delta(p)<0$ when $\gamma_p$ projects onto a circle of $\mathbb{H}^2(\kappa)$. This agrees with the three different kinds of geodesics in $\mathbb{E}(\kappa,\tau)$, $\kappa<0$, whose explicit parameterizations can be found in~\cite{MN}.
\end{remark}

A straightforward calculation shows that the solution to the initial value problem \eqref{eq:ode}-\eqref{eq:initial_conditions} is given by
\begin{align} \nonumber
b_{11}(t)={}&1+4\delta^{-1} \tau ^2 a_{11}s_\delta(t) - 2 \delta^{-1}\tau  \left(a_{12}+\tau \right)(c_\delta(t)-1) - a_{11} \delta^{-1}
(\delta +4 \tau ^2)t,
\\ \nonumber
b_{21}(t)={}& 2 \tau\delta^{-1}  a_{11} \bigl(c_\delta(t)-1\bigr)- \bigl(a_{12}+\tau \bigr) s_\delta(t),
\\ \nonumber
b_{12}(t)={}&  \delta^{-1}\left(2 \tau    (2 \tau 
a_{12}+\delta +2 \tau ^2)s_\delta(t) -2 \tau   
a_{22} (c_\delta(t)-1)
-  (\delta +4 \tau ^2)
(a_{12}+\tau)t\right),
\\ \label{eq:b_ij}
b_{22}(t) ={}&1-a_{22} s_\delta(t)+\delta^{-1} (2 \tau a_{12}+\delta +2 \tau ^2)(c_\delta(t)-1),
\end{align}
where we have considered the auxiliary functions
\begin{equation*}
s_\delta(t)=\begin{cases}\frac{1}{\sqrt{\delta}}\sinh(t\sqrt{\delta})&\text{if }\delta>0,\\
\frac{1}{\sqrt{-\delta}}\sin(t\sqrt{-\delta})&\text{if }\delta<0,
\end{cases}\qquad 
c_\delta(t)=\begin{cases}\cosh(t\sqrt{\delta})&\text{if }\delta>0,\\
\cos(t\sqrt{-\delta})&\text{if }\delta<0.
\end{cases}
\end{equation*}

Now, for each $r\in(-\varepsilon,\varepsilon)$, we consider the linear endomorphisms $B(r)$ and $C(r)$ of $T_{\gamma_p(r)}\Sigma^r$ determined by the relations
\[
B(r)U_j(\gamma_p(r))=\zeta_j(r), \qquad C(r)U_j(\gamma_p(r))=\zeta_j'(r), \qquad i\in\{1,2\}.
\]
Then, Jacobi field theory (see~\cite[Theorem~10.2.1]{BCO}) guarantees that, since $\Sigma^r$ has dimension two, $\zeta_j(r)\neq 0$ for each $r\in(-\varepsilon,\varepsilon)$, and also that the shape operator $A^r$ of the equidistant surface $\Sigma^r$ with respect to the normal vector $\gamma_p'(r)$ is determined by $A^r\zeta_j(r)=-\zeta_j'(r)$.  Thus, the operator $B(r)$ is invertible for each $r\in(-\varepsilon,\varepsilon)$, and the shape operator $A^r$ is given by 
\begin{equation}\label{eq:A^r}
A^r=-C(r)B(r)^{-1}.
\end{equation}
With respect to the orthonormal basis $\{U_1(\gamma_p(r)),U_2(\gamma_p(r)) \}$ of $T_{\gamma_p(r)}\Sigma^r$, the operators $B(r)$ and $C(r)$ take the matrix forms
\begin{equation}\label{eq:BC}
B(r)=\begin{pmatrix}
b_{11}(r) & b_{12}(r)
\\
b_{21}(r) & b_{22}(r)
\end{pmatrix},
\;
C(r)=
\begin{pmatrix}
b_{11}'(r) +\tau b_{21}(r)& b_{12}'(r)+\tau b_{22}(r)
\\
b_{21}'(r) - \tau b_{11}(r) & b_{22}'(r) - \tau b_{12}(r)
\end{pmatrix},
\end{equation}
where the expression for $C(r)$ follows from~\eqref{eq:nabla_u_i}. By \eqref{eq:A^r} and~\eqref{eq:BC}, a direct computation shows that the mean curvature of $\Sigma^r$ is given by the function
\begin{equation*}
h(r)=\frac{1}{2}\mathop{\rm tr}  A^r=-\frac{1}{2}\mathop{\rm tr}\left[C(r)B(r)^{-1}\right]=-\frac{\frac{d}{dr}(\det B(r))}{2\det B(r)}.
\end{equation*}
Hence, the real function $f(r)=\frac{d}{dr}(\det B(r))+2h(r)\det B(r)$ vanishes identically. Using~\eqref{eq:b_ij} to compute $\det B(r)$ explicitly, some elementary calculations show that
\begin{align} \nonumber
\delta^2f(r)={}&\delta\bigl(\delta^2+3\delta\tau^2+4\tau^4+2\tau(\delta+4\tau^2)a_{12}
 + (\delta-4\tau^2)(a_{11}a_{22}-a_{12}^2)\bigr)  s_\delta(r)
\\ \nonumber
&-\delta^2(a_{11}+a_{22})c_\delta(r)
-\delta^2(\delta+4\tau^2)a_{11}r s_\delta(r)
\\ \nonumber
&
+\delta(\delta+4\tau^2)(a_{11}a_{22}-(\tau+a_{12})^2)r c_\delta(r)
\\ \nonumber
&-2\delta(\delta+4\tau^2)((\tau+a_{12})^2-a_{11}a_{22})rh(r)s_{\delta}(r)
\\ \nonumber
&-2\delta(\delta+4\tau^2)a_{11}rh(r)c_\delta(r)
-2\delta(\delta a_{22}-4\tau^2a_{11})h(r)s_\delta(r)
\\ \nonumber
&+2\bigl((\delta+4\tau^2)(\delta+4\tau a_{12})-8\tau^2(a_{11}a_{22}-a_{12}^2-\tau^2)\bigr) h(r)c_\delta(r)
\\ \label{eq:f}
&+8\tau\bigl(2\tau a_{11}a_{22}-(\tau+a_{12})(\delta+2\tau^2+2\tau a_{12})\bigr)h(r).
\end{align}

Note that $\delta$ and the entries $a_{ij}$ of the shape operator $A$ of $\Sigma$ are functions of the base point $p\in\Sigma$. However, by assumption, $\Sigma$ is isoparametric and, hence, the mean curvature $h(r)$ of the equidistant surface $\Sigma^r$ is constant, that is, it is independent of the chosen base point $p\in \Sigma$ of the normal geodesic $\gamma_p$. We will show that, in this case, the functions $\delta$ and $a_{ij}$ are independent of $p\in\Sigma$, which will prove that $\Sigma$ has constant principal curvatures, and also constant angle function because of~\eqref{eq:delta}.

Since $h(0)$ is the mean curvature of $\Sigma$, we have
\begin{equation}
\label{eq:a_22}
a_{22}=2h(0)-a_{11}.
\end{equation}
As $f\equiv 0$, all derivatives of $f$ at $0$ vanish. Taking derivatives in~\eqref{eq:f}, and using that $s_\delta'=c_\delta$ and $c_\delta'=\delta s_\delta$, we obtain the following relations:
\begin{align}
\label{eq:f'(0)}
0=f'(0)={}&\delta+2\tau^2-2h(0)(a_{11}+a_{22})+2a_{11}a_{22}-2a_{12}^2+2h'(0)
\\ \nonumber
0=f''(0)={}&2\bigl(\delta h(0)+2\tau^2h(0)+h''(0)\bigr)-4h(0)a_{12}^2
\\ \label{eq:f''(0)}
&+4h(0)a_{11}a_{22}-(\delta+4h'(0))a_{22}-\bigl(3\delta+8\tau^2+4h'(0)\bigr)a_{11}
\\\nonumber
0=f'''(0)={}&\delta^2-8\tau^4+6(\delta+2\tau^2)h'(0)+2h'''(0)-6h''(0)(a_{11}+a_{22})
\\\nonumber
&-2\delta h(0) a_{22}-4\tau(\delta+4\tau^2)a_{12}-4\bigl(\delta+2\tau^2+3h'(0)\bigr)a_{12}^2
\\ \label{eq:f'''(0)}
&-2 (3\delta+8\tau^2)h(0)a_{11}+4 \bigl(\delta+2\tau^2+3h'(0)\bigr)a_{11}a_{22}.
\end{align}
By changing the orientation of the unit normal $N$ to $\Sigma$ if necessary, we can assume that $a_{12}\geq 0$. Then, inserting \eqref{eq:a_22} into~\eqref{eq:f'(0)} and solving for $a_{12}$, we get
\begin{equation}\label{eq:a_12}
a_{12}=\frac{1}{\sqrt{2}}\sqrt{\delta+2\tau^2-4h(0)^2+2h'(0)+4h(0)a_{11}-2a_{11}^2}.
\end{equation}
Because of~\eqref{eq:delta} and the assumption that $\nu^2\neq 1$ on $\Sigma$, we have that $\delta+4\tau^2=(\nu^2-1)(\kappa-4\tau^2)\neq 0$ on $\Sigma$. Then, using \eqref{eq:a_22} and \eqref{eq:a_12} in \eqref{eq:f''(0)}, we get that
\begin{equation}\label{eq:a_11}
a_{11}=\frac{1}{\delta+4\tau^2}\left(4h(0)^3-6h(0)h'(0)+h''(0)-h(0)\delta\right).
\end{equation}
Finally, substituting~\eqref{eq:a_22}, \eqref{eq:a_12} and~\eqref{eq:a_11} into~\eqref{eq:f'''(0)}, after some calculations we obtain a relation of the form 
\[
q_1(\delta)\pm\tau\sqrt{q_2(\delta)}=0,
\]
where $q_1$ and $q_2$ are polynomials of degrees $2$ and $3$, respectively, with constant coefficients. Hence, if $\tau=0$, $\delta$ satisfies a polynomial equation with constant coefficients and degree $2$ and, if $\tau\neq0$, getting rid of the radical in the previous expression, we obtain that $\delta$ satisfies a polynomial equation with constant coefficients and degree~$4$. In any case, if follows that $\delta$ is constant and, because of the relations~\eqref{eq:a_22}, \eqref{eq:a_12} and \eqref{eq:a_11}, the principal curvatures of $\Sigma$ are constant, which concludes the~proof.

\end{document}